\documentclass{amsart}
\usepackage{latexsym, amsmath, amsfonts, amssymb, geometry, mathrsfs, fancyhdr, tikz, tikz-cd}
\usetikzlibrary{matrix,arrows,decorations.pathmorphing,calc}

\def\n{\noindent}
\def\R{\mathbb{R}}

\def\E{\mathbb{E}}

\def\v{\varphi}
\def\e{\varepsilon}
\def\s{\sigma}

\def\e{\epsilon}
\def\d{\delta}

\def\T{\mathcal{T}}
\def\P{\mathcal{P}}

\def\En{\mathbb{E}^n}

\def\X{\mathcal{X}}
\def\Y{\mathcal{Y}}
\def\V{\mathcal{V}}

\def\a{\alpha}
\def\b{\beta}

\def\dx{d_{\X}}
\def\dy{d_{\Y}}
\def\pull{\text{pull}}

\newtheorem{theorem}{Theorem}[section]
\newtheorem{lemma}[theorem]{Lemma}
\newtheorem{cor}[theorem]{Corollary}

\theoremstyle{definition}

\theoremstyle{remark}

\numberwithin{equation}{section}

\begin{document}

\title{Isometric Embeddings of Polyhedra into Euclidean Space}

\author{B. Minemyer}
\address{Division of Mathematics, Alfred University, Alfred, New York 14802}
\email{minemyer@alfred.edu}



\date{November 27, 2013.}



\begin{abstract}

In this paper we consider piecewise linear (pl) isometric embeddings of Euclidean polyhedra into Euclidean space.  A Euclidean polyhedron is just a metric space $\P$ which admits a triangulation $\T$ such that each $n$-dimensional simplex of $\T$ is affinely isometric to a simplex in $\mathbb{E}^n$.  We prove that any 1-Lipschitz map from an $n$-dimensional Euclidean polyhedron $\P$ into $\mathbb{E}^{3n}$ is $\e$-close to a pl isometric embedding for any $\e > 0$.  If we remove the condition that the map be pl then any 1-Lipschitz map into $\mathbb{E}^{2n + 1}$ can be approximated by a (continuous) isometric embedding.  These results are extended to isometric embedding theorems of spherical and hyperbolic polyhedra into Euclidean space by the use of the Nash-Kuiper $C^1$ isometric embedding theorem (\cite{Nash1} and \cite{Kuiper}).  Finally, we discuss how these results extend to various other types of polyhedra.

\end{abstract}

\maketitle



\section{History and Introduction}\label{History and Introduction}

In \cite{Zalgaller}, Zalgaller proves that every Euclidean polyhedron\footnote{In the literature, a Euclidean polyhedron is sometimes referred to as a Polyhedral Space.  We use the term Euclidean Polyhedron to be consistent with \cite{Petrunin1}} with dimension $n =$ 2 or 3 admits a pl isometry into $\mathbb{E}^n$.  This, along with the Nash isometric embedding theorems \cite{Nash1} and \cite{Nash2}, prompted Gromov in \cite{Gromov2} to ask whether or not this result could be extended to polyhedra of arbitrary dimension.  Answering this question in the affirmitive was Krat in \cite{Krat}, but here she asked a more subtle question.  Namely, can any short\footnote{1-Lipschitz} map from an $n$-dimensional Euclidean polyhedron into $\mathbb{E}^n$ be approximated by a pl isometry?  Krat proved that this question was true in the case when $n=2$, and Akopyan in \cite{Akopyan} generalized Krat's work to arbitrary dimensions.

The above results of Zalgaller, Krat, and Akopyan are striking because of the equality of the dimensions involved.  At first glance it may even seem that these results cannot possibly be true.  But in \cite{Burago} Burago gives a good example making the dimension 2 case more intuitive.  His example is to use paper, scissors, and glue to construct your favorite 2-dimensional Euclidean polyhedron\footnote{Of course, not every 2-dimensional Euclidean polyhedron can be constructed in this way.  But this example still gets the point across.}.  Then set it on the floor and step on it.  This is a pl isometry into $\mathbb{E}^2$.

The purpose of this paper is to investigate the questions of Krat and Akopyan for the case of isometric embeddings, and not merely just isometries.  Obviously, the dimension requirements will have to increase.  But at first glance it is not obvious by how much the dimensions must increase, or even if such isometric embeddings exist.  As we will see in section \ref{Proof of the Main Theorem}, the fact that such isometric embeddings exist is essentially a corollary of Akopyan's result.  Namely, we will observe the following:

\begin{cor}[Akopyan]\label{Akopyan Corollary}

Let $\P$ be an $n$-dimensional Euclidean polyhedron and let $f:\P \rightarrow \mathbb{E}^{3n + 1}$ be a short map.  Then $f$ is $\e$-close to a pl intrinsic\footnote{Since the spaces considered are proper geodesic metric spaces, by Le Donne in \cite{Donne} intrinsic isometries are equivalent to the more common notion of a path isometry.  In any case, this will be defined in Section \ref{Preliminaries}.} isometric embedding for any $\e > 0$.

\end{cor}

In light of Corollary \ref{Akopyan Corollary}, our first goal is to lower the dimension of the target Euclidean space.  The ultimate goal would be to get the dimension all the way down to $2n+1$.  The first main result of this paper uses a slight modification of Akopyan's result in conjunction with a topological trick to reduce this dimension by 1 to $3n$.  It is then shown that, if we remove the restriction that our approximate isometric embedding be pl, we may lower the dimensionality of our target Euclidean space all the way down to our goal of $2n + 1$.  These results are as follows:

\begin{theorem}\label{Main Theorem - Euclidean}

Let $\P$ be an $n$-dimensional Euclidean polyhedron, let $f:\P \rightarrow \E^N$ be a short map, and let $\{ \e_i \}_{i = 1}^{\infty}$ be a sequence of positive real numbers.  Fix a triangulation $\T$ of $\P$ and let $v$ be a fixed vertex of $\T$.  Then

	\begin{enumerate}
	
	\item There exists a pl intrinsic isometric embedding $h: \P \rightarrow \mathbb{E}^N$ which is an $\e_l$ approximation of $f$ within $Sh^l(v)$ provided $N \geq 3n$.
	
	\item There exists an (continuous) intrinsic isometric embedding $h: \P \rightarrow \mathbb{E}^N$ which is an $\e_l$ approximation of $f$ within $Sh^l(v)$ provided $N \geq 2n + 1$.
	
	\end{enumerate}

\end{theorem}

\begin{cor}

Let $\P$ be an $n$-dimensional Euclidean polyhedron.  Then $\P$ admits a pl intrinsic isometric embedding into $\E^{3n}$ and a continuous intrinsic isometric embedding into $\E^{2n+1}$.

\end{cor}

Theorem \ref{Main Theorem - Euclidean} will be proved in Section \ref{Proof of the Main Theorem}, and all necessary terminology will be defined in Sections \ref{Preliminaries} and \ref{Key Lemmas}.

Akopyan's results also hold for spherical and hyperbolic polyhedra (with a slight caveat necessary for spherical polyhedra).  In a similar fashion, the proof of Theorem \ref{Main Theorem - Euclidean} (2) also goes through for spherical and hyperbolic polyhedra.  This is part (2) below.  Unfortunately, our proof for Theorem \ref{Main Theorem - Euclidean} (1) does not go through perfectly.  It does yield a somewhat interesting analogue for the general curvature case though, which is (1) below.  But what is neat about the general curvature case is that we can use the Nash-Kuiper $C^1$ isometric embedding Theorem (see \cite{Nash1} and \cite{Kuiper}) to obtain the third part of the following:

\begin{theorem}\label{Main Theorem - General Curvature}

Let $\P$ be an $n$-dimensional polyhedron with curvature $k$, let $f:\P \rightarrow \mathbb{M}^N_k$ be a short map (which, in the case when $k > 0$, is not surjective), and let $\{ \e_i \}_{i = 1}^{\infty}$ be a sequence of positive real numbers.  Fix a triangulation $\T$ of $\P$ and let $v$ be a fixed vertex of $\T$.  Then

	\begin{enumerate}
	
	\item There exists a pl intrinsic isometric embedding of $\P$ into either $\mathbb{M}^{2n}_k \times \mathbb{M}^n_k$ if $k \geq 0$ or $\mathbb{M}^{2n + 1}_k \times \mathbb{M}^n_k$ if $k<0$.
	
	\item There exists an (continuous) intrinsic isometric embedding $h: \P \rightarrow \mathbb{M}^N_k$ which is an $\e_l$ approximation of $f$ within $Sh^l(v)$ provided $N \geq 2n + 1$ if $k \geq 0$ or $N \geq 2n+2$ if $k < 0$.
	
	\item Every $n$-dimensional polyhedron of curvature $k$ admits an intrinsic isometric embedding into $\mathbb{E}^{2n + 2}$ if $k \geq 0$ or into $\E^{2n+3}$ if $k<0$.
	
	\end{enumerate}

\end{theorem}

Theorem \ref{Main Theorem - General Curvature} will be proved in Section \ref{Proof of the Main Theorem}.

All of the results above apply to polyhedra which are locally finite.  But there are many places where mathematicians study polyhedra which are not locally finite, but rather have some other sort of ``compactness condition".  One very common example is that the polyhedron has only ``finitely many isometry types of simplices".  Such polyhedra are studied, for example, by Martin Bridson in \cite{Bridson} and \cite{BH}.  The above results apply to these types of polyhedra as well, and this is the topic of Section \ref{Polyhedra with Finitely Many Isometry Types of Simplices}.

Finally, additional improvements on the dimension of the target space can be made for polyhedra which are highly connected.  These improvements will be discussed in Section \ref{Improvements for Highly Connected Polyhedra}.

\subsection*{Acknowledgements} The author wants to thank Pedro Ontaneda, Ross Geoghegan, Tom Farrell, Mladen Bestvina, Anton Petrunin, and many others for helpful remarks and guidance during the writing of this article.  In particular, an email from Anton Petrunin and his unpublished lecture notes \cite{Petrunin2} were essential in helping the author locate and understand the contents of papers that only existed in Russian (see \cite{Akopyan} and \cite{Zalgaller}).  During the preparation of this paper the author received tremendous support from both Binghamton University and Alfred University.  This research was partially supported by the NSF grant of Tom Farrell and Pedro Ontaneda, DMS-1103335.

\section{Preliminaries}\label{Preliminaries}

\subsection{Polyhedra of Curvature $k$}

A \emph{polyhedron} $\P$ \emph{with (local) curvature $k$} is a metric space which admits a triangulation\footnote{All triangulations in this paper are simplicial} so that every $l$-dimensional simplix of the triangulation is affinely isometric to a simplex in either Euclidean space $\E^l$ (if $k = 0$), the $l$-sphere $\mathbb{S}^l_k$ of curvature $k$ (if $k > 0$), or hyperbolic space $\mathbb{H}^l_k$ of curvature $k$ (if $k < 0$).  If $k = 0$ we call $\P$ a \emph{Euclidean polyhedron}, if $k > 0$ we call $\P$ a \emph{spherical polyhedron}, and if $k < 0$ we call $\P$ a \emph{hyperbolic polyhedron}.  Until Section \ref{Polyhedra with Finitely Many Isometry Types of Simplices} we will assume that all polyhedra $\P$ admit a triangulation which is locally finite.  $\P$ has dimension $n$ if the maximal dimension of any simplex of a triangulation of $\P$ is $n$.  

Throughout the rest of this paper, $\mathbb{M}^n_k$ denotes the $n$-dimensional model space of curvature $k$.  More precisely, $\mathbb{M}^n_k$ is $\En$ if $k = 0$, $\mathbb{M}^n_k$ denotes $\mathbb{S}^n_k$ if $k > 0$, and $\mathbb{M}^n_k$ denotes $\mathbb{H}^n_k$ if $k < 0$.

\subsection{General Position}

A set of points in $\mathbb{R}^N$ is said to be in \emph{general position} if no $l + 1$ points lie on an $l - 1$ dimensional affine subspace for any $1 \leq l \leq N$.  Suppose $k$ and $N$ are integers with $k \leq N$.  A set of points in $\mathbb{R}^N$ is said to be in \emph{$k$-general position} if no $l + 1$ points lie on an $l - 1$ dimensional affine subspace for any $1 \leq l \leq k$.

Analoguously, if we think of $\mathbb{S}^N_k$ as a subspace of $\R^{N + 1}$ and if we think of points of $\mathbb{S}^N_k$ as vectors in $\R^{N + 1}$ whose initial point is the origin, a collection of points in $\mathbb{S}^N_k$ is said to be in \emph{general position} if every collection of no more than $N + 1$ points corresponds to a collection of vectors which is linearly independent.  We apply the same definition to points in $\mathbb{H}^N_k$ where we think of $\mathbb{H}^N_k$ as a subset of $\mathbb{R}^N$ using the upper half-plane model\footnote{Of course, this doesn't preserve the metric of $\mathbb{H}^N_k$.  But the metric plays no role in the definition of general position}.

An important Lemma, whose proof is contained in \cite{HY}, is the following.

\begin{lemma}\label{Key Lemma}

Let $\P$ be an $n$-dimensional Euclidean polyhedron with a fixed triangulation $\T$ whose vertex set is $\V$, and let $f: \P \rightarrow \E^N$ be a simplicial map (with respect to $\T$).  Let $f(\V)$ denote the collection of images of the vertices of $\T$.  If $f(\V)$ is in ($2n + 1$)-general position (so in particular we must have $N$ $\geq$ $2n + 1$) then $f$ is an embedding.

\end{lemma}

\begin{cor}[Corollary of the proof of Lemma \ref{Key Lemma}]\label{Key Corollary}

Let $\P$, $\T$, $\V$, $f$ and $f(\V)$ be as in Lemma \ref{Key Lemma}.  If $f(\V)$ is in ($2n$)-general position then, for all $p \in \P$, $f|_{St(p)}$ is an embedding\footnote{Where $St(p)$ denotes the closed star of $p$ with respect to the triangulation $\T$.}.    

\end{cor}

As stated, Lemma \ref{Key Lemma} and the resulting Corollary only deal with Euclidean polyhedra.  In order to prove Theorem \ref{Main Theorem - General Curvature} we will need the following analogue for spherical and hyperbolic polyhedra:

\begin{lemma}\label{Key Lemma - Spherical and Hyperbolic Polyhedra}

Let $\P$ be an $n$-dimensional polyhedron of curvature $k$ with a fixed triangulation $\T$ whose vertex set is $\V$, and let $f: \P \rightarrow \mathbb{M}^N_k$ be a simplicial map (with respect to $\T$).  Let $f(\V)$ denote the collection of images of the vertices of $\T$.  If $f(\V)$ is in ($2n + 1$)-general position and $k > 0$ then $f$ is an embedding.  If $f(\V)$ is in ($2n + 2$)-general position and $k < 0$ then $f$ is an embedding.

\end{lemma}

\begin{cor}[Corollary of the proof of Lemma \ref{Key Lemma - Spherical and Hyperbolic Polyhedra}]\label{Key Corollary - Spherical and Hyperbolic Polyhedra}

Let $\P$, $\T$, $\V$, $f$ and $f(\V)$ be as in Lemma \ref{Key Lemma - Spherical and Hyperbolic Polyhedra}.  If $f(\V)$ is in ($2n$)-general position and $k > 0$ then, for all $p \in \P$, $f|_{St(p)}$ is an embedding.  If $f(\V)$ is in ($2n + 1$)-general position and $k < 0$ then, for all $p \in \P$, $f|_{St(p)}$ is an embedding.    

\end{cor}

We prove Lemma \ref{Key Lemma - Spherical and Hyperbolic Polyhedra}, the resulting Corollary follows directly.  The following proof is very similar to the Euclidean case from \cite{HY}

\begin{proof}[Proof of Lemma \ref{Key Lemma - Spherical and Hyperbolic Polyhedra}]

Let $\P$ be an $n$-dimensional polyhedra of curvature $k$ and let $n_k$ be either $2n+1$ if $k>0$ or $2n+2$ of $k<0$.  Until the end of the proof we will treat both cases the same.  Let $f: \P \rightarrow \mathbb{M}^{N}_k$ be a simplicial map which maps the vertices of $\T$ to points which are in $n_k$ general position (so $N \geq n_k$).  It is clear that $f$ is an embedding when restricted to any simplex of $\T$.  So suppose $f$ is not an embedding.  Then there exists $x, y \in \P$ such that $f(x) = f(y)$, and thus we have that $x$ and $y$ are in different simplices of $\T$.

Let $\Delta_x$ ($\Delta_y$) denote the unique simplex of $\T$ containing $x$ ($y$) in its interior (where we consider a vertex to be interior to itself).  Let $i$ denote the dimension of $\Delta_x$ and similarly $j$ for $\Delta_y$.  Denote the vertices of $\Delta_x$ by $<v_0, ..., v_i>$ and of $\Delta_y$ by $<w_0, ..., w_j>$.  Thinking of $\mathbb{M}^{N}_k$ as living in $\mathbb{R}^M$ (for either $M = N$ if $k < 0$ or $M = N + 1$ if $k > 0$) we can treat $f(v_0), ..., f(v_i), f(w_0), ..., f(w_j)$ as vectors whose initial point is the origin.  Since $f(x) = f(y)$ we must have that this collection of $i + j + 2$ vectors is linearly dependent.  So we have $i + j + 2 \leq 2n + 2$ vertices whose images under $f$ are not affinely independent.

If $k > 0$, then $M = N + 1 \geq n_k + 1 = 2n + 2$.  If $k > 0$, then $M = N \geq n_k = 2n + 2$.  So in either case, the images of at most $2n + 2$ vertices not being in general position contradict our assumption on $f$.  Therefore, the map $f$ is an embedding.

\end{proof}

\subsection{Pullback Metrics and Intrinsic Isometries}

What follows is almost directly from \cite{Petrunin1}.

Let $(\X, \dx)$ and $(\Y, \dy)$ be metric spaces and $f:\X \rightarrow \Y$ a continuous map.  $f$ is \emph{short} if for any two points $x, x' \in \X$ we have that $d_{\Y}(f(x), f(x')) \leq d_{\X}(x, x')$, and $f$ is \emph{strictly short} if $d_{\Y}(f(x), f(x')) < d_{\X}(x, x')$ for any $x, x' \in \X$ with $x \neq x'$.  

Now, given two points $x, x' \in \X$, a sequence of points $x = x_0, x_1, ..., x_{k - 1}, x_k = x'$ is called an \emph{$\e$-chain} from $x$ to $x'$ if $\dx(x_{i - 1}, x_i) \leq \e$ for any $i$.  Define:

$$ \text{pull}_{f,\e}(x, x') := \text{inf} \left\{ \sum_{i = 1}^{k} \dy(f(x_{i - 1}), f(x_i))   \right\} $$ 

\n where the infimum is taken over all $\e$-chains from $x$ to $x'$.  It is not hard to see that for any $\e > 0$ pull$_{f, \e}$ is almost a metric on $\X$.  The only issue is that we may have pull$_{f, \e} (x, x') = 0$ for some $x \neq x'$.  But pull$_{f, \e}$ is clearly monotone nonincreasing with respect to $\e$.  So it makes sense to define the limit

$$ \text{pull}_f (x, x') := \lim_{\e \to 0} \text{pull}_{f, \e}(x, x') $$ 

\n where this limit may be infinite.  We call pull$_f$ the \emph{pullback metric} for $f$. 

A map $f: \X \rightarrow \Y$ is an \emph{intrinsic isometry} if

$$ \dx(x, x') = \text{pull}_f(x, x') $$

\n for all $x, x' \in \X$. 

In Section \ref{Proof of the Main Theorem} we will need the following Lemma which is proved in \cite{Petrunin1}.

\begin{lemma}\label{Petrunin}

Let $\X$ and $\Y$ be metric spaces with $\X$ compact and let a continuous map $f:\X \rightarrow \Y$ be such that

$$ \sup_{x, x' \in \X} \pull_f(x, x') < \infty. $$

\n Then for any $\e > 0$ there exists $\d = \d(f,\e) > 0$ such that for any short map $h:\X \rightarrow \Y$ satisfying 

$$ \dy(f(x), h(x)) < \d \text{ for any } x \in \X $$
\n we have that

$$ \pull_f(x, x') < \pull_h(x, x') + \e $$

\n for any $x, x' \in \X$

\end{lemma}

\subsection{A Slight Modification of Akopyan's Result}

In \cite{Akopyan} Akopyan proves the following Theorem:

\begin{theorem}[Akopyan]\label{Akopyan}

Let $\P$ be an $n$-dimensional polyhedron with curvature $k$.  Fix a triangulation $\T$ of $\P$ with vertex set $\V$ and let $\e > 0$.  Let $f:\P \rightarrow \mathbb{M}^N_k$ be a short map with $N \geq n$.  If $k > 0$ assume that the map $f$ is not surjective\footnote{This condition is necessary.  See \cite{Akopyan} for an example.}.  Then there exists a pl intrinsic isometry $h:\P \rightarrow \mathbb{M}^N_k$ such that $|f(x) - h(x)| < \e$ for all $x \in \P$.  

\end{theorem}

In order to prove Theorem \ref{Main Theorem - Euclidean} (1) we need a slight stronger version of this result.  The problem with Theorem \ref{Akopyan} is that the $\e$-approximation is uniform across the whole polyhedron.  We need a relative version of this result which allows for $\e$ to taper to 0 as you move away from some fixed point of the polyhedron.  The statement that we need is below.  But before we state the version of Akopyan's Theorem that is needed, some terminology must first be introduced.

Let $\P$ be a polyhedron and let $x \in \P$.  Fix a triangulation $\T$ of $\P$.  For a vertex $v$ the closed star of $v$ will be denoted by $St(v)$.  We define $St^2(v) := \bigcup_{u \in St(v)} St(u)$ and for any $k \in \mathbb{N}$ we recursively define $St^{k + 1}(v) := \bigcup_{u \in St^k(v)} St(u)$.  Then define the \emph{$k^{th}$ shell about $x$}, denoted by $Sh^k(x)$, recursively as:
		
			\begin{enumerate}
			
			\item $Sh^1(x) = St(x)$
			
			\item $Sh^k(x) = St^k(x) \setminus St^{k - 1}(x)$ for $k \geq 2$
			
			\end{enumerate}
			
	Notice that $Sh^k(x) \cap Sh^l(x) = \emptyset$ for $k \neq l$ and that $\bigcup_{i = 1}^{\infty} Sh^i(x) = \P$.  Also note that $St^k(x)$ and $Sh^k(x)$ both depend on the triangulation that we are considering.  If we want to emphasize the triangulation, then we will put it as a subscript.  So $St^k_{\T}(x)$ and $Sh^k_{\T}(x)$ denote the $k^{th}$ closed star and the $k^{th}$ shell of $x$ with respect to $\T$, respectively.  The relative version of Akopyan's Theorem is as follows:

\begin{theorem}[Relative Version of Akopyan's Theorem]\label{Relative Akopyan}

Let $\P$ be an $n$-dimensional polyhedron with curvature $k$.  Fix a triangulation $\T$ of $\P$ with vertex set $\V$ and let $\{ \e_i \}_{i = 1}^{\infty}$ be a sequence of positive real numbers.  Let $f:\P \rightarrow \mathbb{M}^N_k$ be a short map with $N \geq n$ and fix a vertex $v \in \V$.  If $k > 0$ assume that the map $f$ is not surjective.  Then there exists a pl intrinsic isometry $h:\P \rightarrow \mathbb{M}^N_k$ such that for any $l \in \mathbb{N}$ and for any $x \in Sh^l(v)$, $|f(x) - h(x)| < \e_l$.  

\end{theorem}  

The proof of Theorem \ref{Relative Akopyan} goes through in essentially the exact the same manner as Theorem \ref{Akopyan} and can be found in \cite{Minemyer2} for the case when $k = 0$.

\section{Key Lemmas}\label{Key Lemmas}

The following Lemma is proved in \cite{Krat} for the case when $k = 0$.  The same proof goes through for arbitrary $k$.

\begin{lemma}\label{Krat}

Let $\P$ be an $n$-dimensional polyhedron with curvature $k$, let $f:\P \rightarrow \mathbb{M}^{N}_k$ (where $N \geq n$) be a short map, and let $\e > 0$.  Then $f$ is $\e$-close to a short piecewise linear map.

\end{lemma}

It is clear that any short map $f$ into Euclidean or hyperbolic space can be approximated by a strictly short map by fixing a point in the image of $f$ and contracting the image of $f$ (slightly) in the direction of this point.  One can perform a similar construction for a map into a sphere as long as the map is not surjective.  We just need to choose a point in the image whose antipodal point is \emph{not} in the image.  This is why we require that the maps in Theorems \ref{Akopyan} and \ref{Main Theorem - Euclidean} are not surjective when $k > 0$.  The next Lemma builds on the preceeding one in the case of embeddings.

\begin{lemma}\label{embedding lemma}

Let $\P$ be an $n$-dimensional Euclidean polyhedron, let $f:\P \rightarrow \E^{N}$ (where $N \geq 2n + 1$) be a short map, and let $\e > 0$.  Then $f$ is $\e$-close to a short piecewise linear embedding.

\end{lemma}

\begin{proof}

By the preceeding Lemma and the comment thereafter we know that we can approximate $f$ by a strictly short pl map $h_0$ with $\frac{\e}{2}$ accuracy.  Let $\T$ be a triangulation of $\P$ so that $h_0$ is linear with respect to $\T$.  Recall from section 1 that we are assuming that $\T$ is locally finite.

Let $(v_i)_{i = 1}^{\infty}$ be an ordering of the vertices.  We will construct a sequence of simplicial functions (with respect to $\T$) $(h_i)_{i = 1}^{\infty}$ from $\P$ to $\E^N$ in such a way that for all $k$ the images of the first $k$ vertices under $h_k$ are in general position, $h_k$ is strictly short for all $k$, and so that the limit converges uniformly to a short embedding $h$.  $h_k$ will be a distance no more than $\frac{\e}{2^{k + 1}}$ from $h_{k - 1}$ and therefore $h$ will be a distance no more than $\e$ from $f$.

The construction of $h_k$ is recursive.  We suppose that $h_{k - 1}$ is defined and we use this to construct $h_k$ (where we consider our short pl map $h_0$ to be the $h_0$ of this sequence).  Let us begin the construction of $h_k$.  Define $h_k(v_i) := h_{k - 1}(v_i)$ for all $i \neq k$.  The work is to decide the value of $h_k(v_k)$.  
		
Let us first consider a single simplex $\s$ of dimension $l$ which contains the vertex $v_k$ in its boundary.  The intrinsic metric on $\P$ induces a quadratic form $G(\s)$ on $\E^l$ associated to $\s$ (see \cite{Minemyer1}) which is positive definite since $\P$ is a Euclidean polyhedron.  Similarly the map $h_{k - 1}$ induces a positive definite quadratic form $G_{k - 1}(\s)$ associated to $\s$, and the fact that $h_{k - 1}$ is a strictly short map when restricted to $\s$ is equivalent to $G(\s) - G_{k - 1}(\s)$ being a positive definite quadratic form.  
		
Now choose some arbitrary value for $h_k(v_k)$ and let $\d_k = |h_{k - 1}(v_k) - h_k(v_k)|_{\E_N}$.  It is easy to see that the quadratic form induced by $h_k$ associated with $\s$ is of the form $G_{k - 1}(\s) + D_k(\s)$ where the form $D_k(\s) \to 0$ as $\d_k \to 0$.  Then for $h_k$ to be strictly short on $\s$ we need the form $G(\s) - (G_{k - 1}(\s) + D_k(\s))$ to be positive definite.  But the collection of positive definite forms on $\E^{l}$ is open and so we can find $D_k(\s)$ close to $\vec{0}$ so that $G(\s) - (G_{k - 1}(\s) + D_k(\s))$ is positive definite.  So if we choose $\d_k$ small enough so that $G(\s) - (G_{k - 1}(\s) + D_k(\s))$ is positive definite then the map $h_k$ will be short on $\s$.
		
Since $\P$ is locally finite, $v_k$ is contained in $m < \infty$ simplices.  So choose $\d_k$ to be the minimum value chosen over all $m$ simplices.  Then $h_k$ is strictly short on every simplex which contains $v_k$.  But $h_k$ agrees with $h_{k - 1}$ on every simplex which does not contain $v_k$ and thus $h_k$ is srictly short over all of $\P$.  

To ensure that $h_k$ is close enough to $h_{k - 1}$ we just make sure that $\d_k < \frac{\e}{2^{k + 1}}$.
		   
Now we finish the construction of $h_k$.  Let $\d_k$ be defined as above.  Then consider $b(h_{k - 1}(v_k), \d_k)$, the open ball of radius $\d_k$ centered at $h_{k - 1}(v_k)$.  For almost any choice of $y \in b(h_{k - 1}(v_k), \d_k)$ the collection $\{ h_{k}(v_1), ..., h_{k}(v_{k - 1}), y \}$ will be in $(2n + 1)$-general position.  So choose $h_k(v_k)$ to be some such point.

This completes the construction of $h_k$.  Notice that the sequence $(h_k)_{k = 1}^{\infty}$ converges uniformly, and the limit will map the vertices of $\P$ into $(2n + 1)$-general position.  Therefore $h$ is an embedding.  $h$ will be (strictly) short on each simplex of $\P$ and will be $\e$-close to $f$. 
		
\end{proof}

Just as the relative version of Akopyan's Theorem \ref{Relative Akopyan} follows directly from the proof of the non-relative case, it is easy to see how to ``tweak" the proof of Lemma \ref{embedding lemma} to make it relative as well.  If, using the notation of the preceeding proof, the vertex $v_k$ is in $Sh^l(v)$, we just require that $\d_k < \text{min} \{ \frac{\e_1}{2^{k + 1}}, \frac{\e_l}{2^{k + 1}} \}$.  This proves:

\begin{cor}[Relative version of Lemma \ref{embedding lemma}]\label{embedding corollary}

Let $\P$ be an $n$-dimensional Euclidean polyhedron and fix a triangulation $\T$ of $\P$ and a vertex $v$ of $\T$.  Let $f:\P \rightarrow \E^{N}$, where $N \geq 2n + 1$, be a short map and let $\{ \e_i \}_{i = 1}^{\infty}$ be a sequence of positive numbers.  Then there exists a short piecewise linear embedding $h$ such that $|f(x) - h(x)| < \e_k$ for all $x \in Sh^k(v)$.

\end{cor}

Everything in Lemma \ref{embedding lemma} and Corollary \ref{embedding corollary} goes through for spherical and hyperbolic polyhedra as well.  We just need to replace Lemma \ref{Key Lemma} with Lemma \ref{Key Lemma - Spherical and Hyperbolic Polyhedra} and, in the case of spherical polyhedra, assume that our starting map $f$ is not surjective so that we can approximate it by a strictly short map.  This yields:

\begin{cor}[General Curvature version of Corollary \ref{embedding corollary}]\label{embedding corollary - general curvature}

Let $\P$ be an $n$-dimensional polyhedron of curvature $k$ and fix a triangulation $\T$ of $\P$ and a vertex $v$ of $\T$.  Let $f:\P \rightarrow \E^{N}$ be a short map, which is not surjective in the case when $k>0$.  Assume $N \geq 2n+1$ if $k > 0$ and $N \geq 2n + 2$ if $N < 0$, and let $\{ \e_i \}_{i = 1}^{\infty}$ be a sequence of positive numbers.  Then there exists a short piecewise linear embedding $h: \P \rightarrow \mathbb{M}^N_k$ such that $|f(x) - h(x)| < \e_l$ for all $x \in Sh^l(v)$.

\end{cor}

\section{Proof of the Main Theorems}\label{Proof of the Main Theorem}

Before proving Theorem \ref{Main Theorem - Euclidean} let us first show how to prove Corollary \ref{Akopyan Corollary} using Theorem \ref{Akopyan} and Lemma \ref{embedding lemma}.

\begin{proof}[Proof of Corollary \ref{Akopyan Corollary}]

Let $f: \P \rightarrow \mathbb{E}^{N}$ be a short map, where $N \geq 3n + 1$.  By Lemma \ref{Krat} we may assume that $f$ is strictly short and linear with respect to some triangulation $\T$ of $\P$.  Define $f_1: \P \rightarrow \mathbb{E}^{2n + 1}$ to be the first $2n + 1$ coordinate functions of $f$ and let $f_2: \P \rightarrow \mathbb{E}^{N - 2n - 1}$ to be the remaining coordinate functions.  Notice that since $N \geq 3n + 1$, $N - 2n - 1 \geq n$.
	
	Let $\s$ be any simplex of $\T$ and let $G_{f}(\s), G_{f_1}(\s),$ and $ G_{f_2}(\s)$ denote the quadratic forms associated with $\s$ induced by $f, f_1$, and $f_2$ respectively.  It is shown in \cite{Minemyer1} that $G_f(\s) = G_{f_1}(\s) + G_{f_2}(\s)$ for all $\s \in \T$.  Let $G(\s)$ be the quadratic form associated to $\s$ which is induced by the intrinsic metric on $\P$.  Since $f$ is strictly short we know that $G(\s)- G_f(\s) = G(\s) - G_{f_1}(\s) - G_{f_2}(\s)$ is positive definite.  So in particular $G(\s) - G_{f_2}(\s)$ is positive definite with respect to all simplices $\s$ and thus yields a piecewise flat metric on $\P$ which is ``larger" than the metric induced by $f_1$.  Therefore by Lemma \ref{embedding lemma} we can approximate $f_1$ (with $\frac{\e}{2}$ accuracy) by a simplicial (with respect to $\T$) embedding $h_1$ such that $(G(\s) - G_{f_2}(\s)) - G_{h_1}(\s)$ is positive definite with respect to every simplex $\s \in \T$. 
	
	Now by Akopyan's Theorem \ref{Akopyan}, since $G(\s) - G_{h_1}(\s)$ is positive definite on every simplex $\s$ of $\T$, we can approximate $f_2$ (with $\frac{\e}{2}$ accuracy) by a map $h_2$ such that $h_2$ is linear with respect to a subdivision $\T'$ of $\T$ and satisfies that $G_{h_2}(\s') = G(\s') - G_{h_1}(\s')$ on every simplex $\s'$ of $\T'$.  Then since $h_1$ is linear with respect to $\T$ and $\T'$ is a subdivision of $\T$, $h_1$ is linear with respect to $\T'$.  Thus our desired pl isometric embedding is $h := h_1 \oplus h_2: \P \rightarrow \E^N$. 

\end{proof}

\begin{proof}[Proof of Theorem \ref{Main Theorem - Euclidean} (1):]

Let $f: \P \rightarrow \E^{N}$ be a short map, where $N \geq 3n$.  By Corollary \ref{embedding corollary} we may assume that $f$ is a simplicial embedding with respect to some subdivision $\T'$ of $\T$.  Let $(f_i)_{i = 1}^{N}$ denote the coordinate functions of $f$. 
	
	The embedding guaranteed by Corollary \ref{embedding corollary} is not quite good enough here.  What we need is to approximate $f$ by a short pl embedding $F$ which also satisfies that the first $2n$ coordinate functions of $F$, when considered as a function into $\mathbb{E}^{2n}$, also maps the vertices of $\T'$ to points which are in general position.  What follows is the technical proof, but the idea is rather simple.  What we do is, one by one, approximate each coordinate function $f_i$ of $f$ with a real valued function $F_i$ in such a way that the first $i$ coordinate functions map the vertices of $\T'$ into general position when considered as a function into $\mathbb{E}^i$.  So $F_1$ maps no two vertices to the same number, $F_2$ is such that the map $F_1 \oplus F_2$ maps no three vertices onto the same line, etc.  
	
	To be specific, what we do is approximate each real-valued function $f_i$ by a real-valued function $F_i$ on $\P$ in such a way that:

	\begin{enumerate}
	
	\item Each $F_i$ is linear with respect to $\T'$.
	
	\item For each $1 \leq j \leq N$ the map $g_j:\P \rightarrow \mathbb{E}^j$ defined by the first $j$ functions $(F_i)_{i = 1}^j$ is strictly short and maps the vertices of $\T'$ to points which are in general position.
	
	\item $F_j$ will be an $\frac{\e_k}{2N}$ approximation of $F_{j - 1}$ within $Sh^k_{\T}(v)$ for all $j$ and $k$.  So the map $F:\P \rightarrow \mathbb{E}^N$ defined by all of the functions $(F_i)_{i = 1}^{N}$ will be $\frac{\e_k}{2}$ close to $f$ within $Sh^k_{\T}(v)$ and strictly short.  
	
	\end{enumerate}
	
	This construction is very similar to that of Lemma \ref{embedding lemma}.  The construction is recursive in both the functions and the vertices.  So let $(v_l)_{l = 1}^{\infty}$ be the vertex set of $\T'$ and suppose that $F_{j - 1}$ has already been defined for some $j$ (the map $f_1$ will serve as $F_0$) and that $F_j$ has been defined on all of the vertices before some $v_l$ ($F_1(v_1) := f_1(v_1)$ so there is no trouble starting the recursion).  Note that this means that we have already defined the map $g_{j - 1}$ satisfying (2) above, and we have defined the map $g_j$ on all of the vertices whose index is less than $l$. 
	
	Consider the real interval $(f_j(v_l) - \d_{jl} ,f_j(v_l) + \d_{jl} )$ for some $\d_{jl} > 0$.  We claim that almost every point in this interval will work for the image of $F_j(v_l)$.  To see this, consider the line segment in $\mathbb{E}^j$ defined by the curve $t \rightarrow (F_1(v_l), ..., F_{j - 1}(v_l), t)$ for $t \in (f_j(v_l) - \d_{jl} ,f_j(v_l) + \d_{jl} )$.  The intersection of this line segment with an affine (proper) subspace $\mathbb{A}$ of $\mathbb{E}^j$ is either empty, a point, or the whole line segment.  But any affine subspace $\mathbb{A}$ which contains the entire line segment also contains the line determined by this line segment.  In particular, if $\mathbb{A}$ contains the entire line segment then $\mathbb{A}$ contains the point $(F_1(v_l), ..., F_{j - 1}(v_l), 0) = (g_{j - 1}(v_l), 0)$.  So if we project $\mathbb{A}$ onto the first $j - 1$ coordinates we see that $g_{j - 1}(v_l)$ lies on some affine subspace $\mathbb{B}$ of $\mathbb{E}^{j - 1}$ whose dimension is 1 less than $\mathbb{A}$.  
	
	Now let $(v_{i_1}, ..., v_{i_m})$ be some collection of $m$ vertices each of whose index is less than $l$ and consider the $m - 1$ dimensional affine subspace $\text{Span}(g_j(v_{i_1}), ..., g_j(v_{i_m})) \subset \mathbb{E}^j$.  We want to define $F_j(v_l)$ so that $g_j(v_l)$ does not lie on this affine subspace (if $m < j + 1$).  The only situation in which we cannot define $F_j(v_l)$ is if $\text{Span}(g_j(v_{i_1}), ..., g_j(v_{i_m}))$ contains the entire line segment $(F_1(v_l), ..., F_{j - 1}(v_l), t)$ for $t \in (f_j(v_l) - \d_{jl} ,f_j(v_l) + \d_{jl} )$.  But if this was the case then there would exist an $m - 2$ dimensional affine subspace in $\mathbb{E}^{j - 1}$ which contained each of the points $(g_{j - 1}(v_{i_1}), ..., g_{j - 1}(v_{i_m}))$, contradicting the inductive hypothesis that the map $g_{j - 1}$ mapped the vertices of $\T'$ to points in $\mathbb{E}^{j - 1}$ which are in general position. 
	
	Thus for any collection of $m$ vertices (with $m < j + 1$) whose indices are less than $l$, the affine subspace $\text{Span}(g_j(v_{i_1}), ..., g_j(v_{i_m}))$ intersects the curve $(F_1(v_l), ..., F_{j - 1}(v_l), t)$ for $t \in (f_j(v_l) - \d_{jl} ,f_j(v_l) + \d_{jl} )$ in at most one point.  Since there are only finitely many such affine subspaces almost any selection for $F_j(v_l)$ in $(f_j(v_l) - \d_{jl} ,f_j(v_l) + \d_{jl} )$ will work to satisfy (2). 
	
	Clearly the map $F := g_N$ will be $\frac{\e_k}{2}$ close to $f$ within $Sh^k_{\T}(v)$ for $\d_{jl} < \text{min} \{ \frac{\e_k}{4N}, \frac{\e_{k+1}}{4N} \}$ (where we are assuming that our vertex $v_l$ considered above is in $Sh^k_{\T}(v)$).  We can make sure our new map is still short at each step in the recursive procedure by an argument identical to that of Lemma \ref{embedding lemma}, so we omit it here.  This completes the construction of the map $F$. 
	
	Notice that since $N \geq 3n \geq 2n + 1$ the map $F$ is a strictly short embedding which is linear with respect to $\T'$.  But the reason that we went through the previous construction is for the following.  Consider the function $h_1:\P \rightarrow \mathbb{E}^{2n}$ defined by the first $2n$ coordinate functions of $F$.  By (2) of the construction, $h_1$ also maps the vertices of $\T'$ to points in $\mathbb{E}^{2n}$ which are in general position.  Therefore by Corollary \ref{Key Corollary}, $h_1$ is an embedding when restricted to the closed star of any point of $\P$. 
	
	Consider the collection $\{ st(p) | p \in Sh^k_{\T}(v) \}$ where $st(p)$ denotes the open star of $p$ with respect to $\T'$.  Since $\T'$ is locally finite we can clearly choose a finite refinement of this collection which covers $Sh^k_{\T}(v)$.  This finite collection has a Lebesgue number $\d_k > 0$.  Let $\Delta_k$ denote the diagonal\footnote{ $\Delta_k = \{ (x, x) | x \in Cl(Sh^k_{\T}(v)) \} $ and $Cl(Sh^k_{\T}(v))$ denotes the closure of $Sh^k_{\T}(v)$} of $Cl(Sh^k_{\T}(v)) \times Cl(Sh^k_{\T}(v))$ and let $b(\Delta_k, \d_k)$ denote the open neighborhood of radius $\d_k$ of $\Delta_k$.  Then $b(\Delta_k, \d_k)^C$ is a closed subset of $Cl(Sh^k_{\T}(v)) \times Cl(Sh^k_{\T}(v))$ and is therefore compact.  Consider the function $\v: b(\Delta_k, \d_k)^C \rightarrow \mathbb{E}$ defined by $\v(x,y) := |F(x) - F(y)|_{\mathbb{E}^N}$.  $\v$ is positive over all of $b(\Delta_k, \d_k)^C$ since $\Delta_k \subset b(\Delta_k, \d_k)$.  Then since $b(\Delta_k, \d_k)^C$ is compact there exists $\mu_k > 0$ such that $\v(x, y) > \mu_k$ for all $(x, y) \in b(\Delta_k, \d_k)^C$. 
	
	Now define $h_2: \P \rightarrow \mathbb{E}^{N - 2n}$ to be the last $N - 2n$ coordinates of $F$.  We apply Theorem \ref{Relative Akopyan} to $h_2$ to obtain a map $\bar{h_2}$ so that $h := h_1 \oplus \bar{h_2}$ is a linear intrinsic isometry on a subdivision $\T''$ of $\T'$ and with $\bar{\e_k} := \text{min} \{ \frac{\e_{k}}{2}, \frac{\e_{k+1}}{2}, \frac{\mu}{3} \}$ accuracy within $Sh^k_{\T}(v)$.  Then $h(x) \neq h(y)$ for any $(x, y) \in b(\Delta_k, \d_k)^C$, and $h(x) \neq h(y)$ for any $(x, y) \in b(\Delta_k, \d_k)$ since $h_1$ is injective on the $\d_k$ neighborhood of every point.
	
\end{proof}

	\begin{figure}
\begin{center}
\begin{tikzpicture}

\draw (-0.2,0) -- (6,0); 
\draw (5.9,0.1) -- (6,0); 
\draw (5.9,-0.1) -- (6,0); 
\draw (0,-0.2) -- (0,4); 
\draw (-0.1,3.9) -- (0,4); 
\draw (0.1,3.9) -- (0,4); 
\draw (6,0)node[right]{$\mathbb{E}^{j-1}$};
\draw (0,4)node[left]{$\mathbb{E}$};

\draw[fill=black!] (4,2.5) circle (0.3ex);
\draw (4,2.5)node[right]{$(g_j(v_l),f_j(v_l))$};
\draw (4,4) -- (4,1);
\draw (4.25,3.95) arc (60:120:0.5cm); 
\draw (4,3.25)node[left]{$\d_{jl}$};
\draw (3.75,1.05) arc (240:300:0.5cm); 

\draw[dotted] (4,1) -- (4,0);
\draw[fill=black!] (4,0) circle (0.3ex);
\draw (4,0)node[below]{$(g_j(v_l),0)$};

\end{tikzpicture}
\end{center}
\caption{Projecting the interval about $(g_j(v_l),f_j(v_l))$ in the last coordinate onto the first $j-1$ coordinates in the proof of Theorem \ref{Main Theorem - Euclidean} (1).}
\label{ninthfig}
\end{figure}
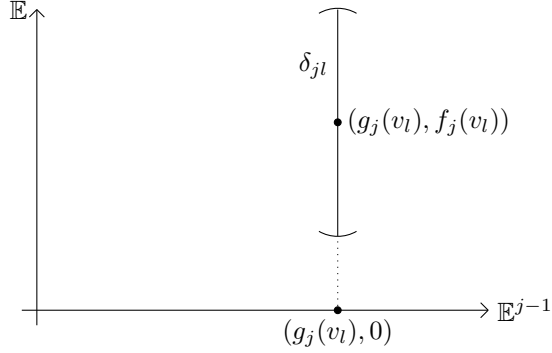

\textbf{Remark: } After beginning the writing of this paper, Tom Farrell showed me a paper of Enrico Le Donne \cite{Donne}.  In it he uses a very similar argument as the second half of the argument above in his proof of Lemma 5.5.

\begin{proof}[Proof of Theorem \ref{Main Theorem - Euclidean} (2):]

Just like the beginning of the last proof, let $f: \P \rightarrow \mathbb{E}^{N}$ be a short map, where $N \geq 2n + 1$.  By Lemma \ref{embedding lemma} we may assume that $f$ is a strictly short embedding which is linear with respect to some subdivision $\T'$ of $\T$.

We will construct two sequences of pl maps $(f_i)_{i = 1}^{\infty}$ and $(h_i)_{i = 1}^{\infty}$.  For any $i$ we will use $h_i$ to construct $f_{i + 1}$ and then use $f_{i + 1}$ to construct $h_{i + 1}$.  $f_i$ will be a strictly short embedding which is an $\a_i^k > 0$ approximation of $h_{i - 1}$ within $Sh^k_{\T}(v)$, and $h_i$ will be an intrinsic isometry which is a $\b_i^k > 0$ approximation of $f_i$ within $Sh^k_{\T}(v)$.  There will exist a sequence of subdivisions $(\T^{i})_{i = 1}^{\infty}$ of $\T'$ such that $h_i$ and $f_{i + 1}$ will be linear with respect to $\T^{i}$.  Our desired map will be $h := \lim_{i \to \infty} h_i$.  $h$ will not be pl since the subdivisions of $\T'$ get finer and finer as $i$ goes to infinity.  The reason for the limit is because, in general, each $f_i$ will not be an isometry and each $h_i$ will not be an embedding.  But for judiciously chosen $\a_{i}^{k}$ and $\b_{i}^{k}$ for each $i$ and $k$ the limit will be an intrinsic isometric embedding which is $\e_k$ close to $f$ within $Sh^k_{\T}(v)$ for each $k$.  The following diagram shows the order in which we construct the maps along with our desired accuracy at each step (here $\a_i := \{\a_i^k \}_{k=1}^{\infty}$ and similarly for $\b_i$).

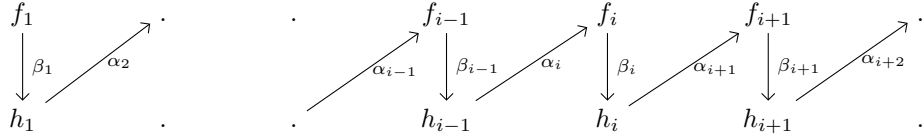
\begin{center}
	\begin{tikzpicture}
\matrix(m)[matrix of math nodes,
row sep=2.6em, column sep=4em,
text height=1.5ex, text depth=0.25ex]
{f_1 & . & . & f_{i-1} & f_i & f_{i+1} & . \\
h_1 & . & . & h_{i-1} & h_i & h_{i+1} & . \\};
\path[->,font=\scriptsize,>=angle 90]
(m-1-1) edge node[right] {$\b_1$} (m-2-1)
(m-2-1) edge node[right] {$\a_2$} (m-1-2)
(m-2-3) edge node[right] {$\a_{i-1}$} (m-1-4)
(m-1-4) edge node[right] {$\b_{i-1}$} (m-2-4)
(m-2-4) edge node[right] {$\a_i$} (m-1-5)
(m-1-5) edge node[right] {$\b_i$} (m-2-5)
(m-2-5) edge node[right] {$\a_{i+1}$} (m-1-6)
(m-1-6) edge node[right] {$\b_{i+1}$} (m-2-6)
(m-2-6) edge node[right] {$\a_{i+2}$} (m-1-7);
	\end{tikzpicture}
	\end{center}
	
We start the procedure by defining $f_1 := f$ which is a strictly short embedding.  We then use Theorem \ref{Relative Akopyan} to define $h_1$ to be an intrinsic isometry which is $\b_1^k$ close to $f_1$ within $Sh^k_{\T}(v)$ and linear with respect to a subdivision $\T^{1}$ of $\T'$.  Assuming $h_{i - 1}$ is defined we use Lemma \ref{embedding lemma} to construct a strictly short embedding $f_i$ which is $\a_i^k$ close to $h_{i - 1}$ within $Sh^k_{\T}(v)$ and linear on the same triangulation $\T^{i - 1}$ that $h_{i - 1}$ is.  We then again invoke Theorem \ref{Relative Akopyan} to construct an intrinsic isometry $h_i$ which is $\b_i^k$ close to $f_i$ within $Sh^k_{\T}(v)$ and linear on a subdivision $\T^i$ of $\T^{i - 1}$.  If at each step $i$ and for each $k$ we choose
	
	 $$\a_i^k, \b_i^k < \frac{\e_k}{4^i}$$
	 
	 \n then the sequence $(h_i)_{i = 1}^{\infty}$ will converge uniformally to a function $h$ which is $\e_k$ close to $f$ within $Sh^k_{\T}(v)$.  So what is left is to show that, at each step, we can choose $\a_i^k$ and $\b_i^k$ small enough so that the limit will be an intrinsic isometric embedding.

We will use two separate tricks.  The trick for the intrinsic isometry is due to Petrunin in \cite{Petrunin1} and the trick for the embedding is due to Nash in \cite{Nash1}.

Let us first show that the limit function $h$ is an embedding.  Enumerate the closed simplices of $\T$ in some way (which is legitimate since $\T$ is locally finite).  Then for all $i \in \mathbb{N}$ define

$$ S_i := \left\{ (x, y) | x \text{ and } y \text{ are both contained in the first } i \text{ simplices and } |f_1(x) - f_1(y)|_{\mathbb{E}^N} \geq 2^{-i} \right\} $$

The set $S_i$, being a closed subset of a compact space\footnote{the product of the union of the first $i$ closed simplices}, is compact.  Since $f_i$ is an embedding the function $\v_i: S_i \rightarrow \mathbb{E}$ defined by $\v_i(x, y) = |f_i(x) - f_i(y)|$ obtains a minimum $\mu_i > 0$.  So by choosing $\b_i^k < \frac{\mu_i}{4}$ for each $k$ and $\a_l^k, \b_l^k < \frac{\mu_i}{4^l}$ for all $l > i$ and for each $k$ we ensure that any pair of points in $S_i$ cannot come together in the limit.  Eventually every pair of distinct points is contained in some $S_i$ and thus the limit function $h$ is an embedding.

Now we consider the problem of the limit being an intrinsic isometry.  Since $\text{pull}_{h_i}(x, x') = d_{\P}(x, x')$ for all $i \in \mathbb{N}$ and for all $x, x' \in \P$ we have that $\lim_{i \to \infty} \text{pull}_{h_i}(x, x') = d_{\P}(x, x')$ for any $x, x'$.  And it is clear that $\text{pull}_h(x, x') \leq \lim_{i \to \infty} \text{pull}_{h_i}(x, x') = d_{\P}(x, x')$, so what we need to show is that $d_{\P}(x, x') \leq \text{pull}_h(x, x')$ for any $x, x' \in \P$.
	
	Just as before we enumerate the closed simplices of $\T$ in some way.  Now we define $U_i$ to be the union of the first $i$ simplices.  Notice that $U_i$ is compact for all $i$ and that $\cup_{i = 1}^{\infty} U_i = \P$.  Let $g_i := h_i|_{U_i}$ and notice that for all $x, x'$ in a connected component\footnote{If $x$ and $x'$ are in different path components of $U_i$ then we will have that $\text{pull}_{g_i}(x, x') = 0 \ngeq \text{pull}_{h_i}(x, x')$} of $U_i$ we have that $\text{pull}_{g_i}(x, x') \geq \text{pull}_{h_i}(x, x')$.  So choose $\d(g_i, \frac{1}{i})$ as in Lemma \ref{Petrunin} and choose $\a_l^k, \b_l^k < \frac{\d(g_i, \frac{1}{i})}{4^l}$ for all $l > i$ and for each $k$.  We will then have that $|h(x) - g_i(x)|_{\mathbb{E}^N} < \d(g_i, \frac{1}{i})$ for all $x \in U_i$ and thus for all $x, x'$ in the same component of $U_i$ we have

$$ \text{pull}_h(x, x') + \frac{1}{i} > \text{pull}_{g_i}(x, x') \geq \text{pull}_{h_i}(x, x') = d_{\P}(x, x'). $$  

Any two points of $\P$ are eventually in the same component of some $U_i$ which completes the proof.

\end{proof}

It is easy to see how to alter the proof of Theorem \ref{Main Theorem - Euclidean} to prove Theorem \ref{Main Theorem - General Curvature} parts (1) and (2).  The only real difference is to replace Corollary \ref{embedding corollary} with Corollary \ref{embedding corollary - general curvature}.  

Notice that in the proof of Theorem \ref{Main Theorem - Euclidean} (1) the general idea is to split the original map $f$ into two maps $f_1$ and $f_2$.  $f_1$ is approximated by a local embedding, and then $f_2$ is approximated by an isometry with respect to the ``remaining" metric.  That is why in Theorem \ref{Main Theorem - General Curvature} (1) the target space is $\mathbb{M}^{2n}_k \times \mathbb{M}^n_k$ ($k>0$) or $\mathbb{M}^{2n+1}_k \times \mathbb{M}^n_k$ ($k<0$) instead of either $\mathbb{M}^{3n}_k$ or $\mathbb{M}^{3n+1}_k$, respectively.  But this problem does not occur in the proof of Theorem \ref{Main Theorem - Euclidean} (2).

What remains is to prove Theorem \ref{Main Theorem - General Curvature} (3).  The trick is to apply the Nash-Kuiper $C^1$ isometric embedding Theorem to Theorem \ref{Main Theorem - General Curvature} (2).  Namely, we will use the following:

\begin{theorem}[Kuiper \cite{Kuiper}]\label{Kuiper}

Let $M$ be an $n$-dimensional Riemannian manifold and let $f:M \rightarrow \mathbb{E}^N$ be a short embedding with $N \geq n + 1$.  Then $f$ is $\e$-close to a $C^1$ isometric embedding for any $\e > 0$.

\end{theorem}

On a historical note, John Nash in \cite{Nash1} proved the above statement but for $N \geq n+2$ and conjectured that it may be true for $N \geq n+1$ using a more controlled ``spiraling" technique.  A few years later Kuiper confirmed Nash's conjecture, and therefore the above Theorem \ref{Kuiper} is often referred to as the ``Nash-Kuiper $C^1$ isometric embedding theorem".

Obviously the $n$-sphere of any curvature $k>0$ admits an (smooth) isometric embedding into $\mathbb{E}^{n + 1}$.  But by the Nash-Kuiper Theorem \ref{Kuiper} hyperbolic $n$-space of any curvature $k<0$ admits a $C^1$ isometric embedding into $\mathbb{E}^{n+1}$.  For the original short embedding just consider the Poincare disk model.  So to prove Theorem \ref{Main Theorem - General Curvature} (3) just compose the isometric embedding from Theorem \ref{Main Theorem - General Curvature} (2) with the isometric embedding guaranteed by the Nash-Kuiper Theorem.

\section{Polyhedra with Finitely Many Isometry Types of Simplices}\label{Polyhedra with Finitely Many Isometry Types of Simplices}

Bridson in \cite{Bridson} and Bridson and Haefliger in \cite{BH} study metric polyhedra with piecewise constant curvature $k$, but the metric polyhedra which they need are \emph{not} locally finite.  So at first glance it seems that none of the results contained in this paper apply to this situation.  But they do impose a different sort of ``compactness" condition, namely that the polyhedron contains only finitely many isometry types of simplices.  What we do now is briefly discuss how we can substitute this condition for local finiteness to maintain all of the preceeding results contained within this paper.

As far as the proofs contained in this paper, it is easy to trace back through to verify that everything still holds.  The first place that we used local finiteness is in the proof of Lemma \ref{embedding lemma} when we chose $\d_k$.  But if there are only finitely many isometry types of simplices then there are only finitely many different restrictions for $\d_k$.  This is exactly the same situation as in the proof of Theorem \ref{Main Theorem - Euclidean} (1), and those are the only places where we used local finiteness.

What is trickier is verifying that this change in assumptions is legitimate for Akopyan's Theorem \ref{Akopyan} and the relative version \ref{Relative Akopyan} as these Theorems were used often throughout this paper.  Akopyan's paper \cite{Akopyan} is only available in Russian\footnote{as far as the author knows}, but an English proof of Theorem \ref{Relative Akopyan} is available in \cite{Minemyer2}.  

The proof of Akopyan's Theorem breaks down into three parts.  The first part is to construct a local projection near every face with codimension at least two.  The second part is to scale the composition of this projection with our given map so that this composition is strictly short.  The third and final step is to apply an approximation method to each simplex using a Theorem due to Brehm in \cite{Brehm}, one skeleton at a time, which builds up the intrinsic isometry.  The first two easily go through if there are finitely many isometry types of simplices.  The third part requires the axiom of choice since the collection of simplices may no longer be countable.  But this is fine, and everything goes through.

To recap, every result in this paper holds if we replace ``locally finite" with ``finitely many isometry types of simplices".

\section{Improvements for Highly Connected Polyhedra}\label{Improvements for Highly Connected Polyhedra}

Combining results from \cite{PWZ} with Theorem \ref{Main Theorem - Euclidean} results in the following two Theorems:

\begin{theorem}

Suppose $M$ is an $n$-dimensional Euclidean Polyhedron which is also a closed $(m-1)$-connected manifold, where $0 < 2m \leq n$.  Then $M$ admits a pl isometric embedding into $\mathbb{E}^{3n - m + 1}$

\end{theorem}

\begin{theorem}

Suppose $M$ is an $n$-dimensional Euclidean Polyhedron which is also a compact $(m-1)$-connected manifold, where $0 < 2m \leq n$.  If either:

	\begin{enumerate}
	
	\item $\partial M \times I$ can be embedded into $\mathbb{R}^{2n-m}$, or
	
	\item $\partial M$ is $(m-2)$-connected\footnote{(-1)-connected is a vacuous condition},
	
	\end{enumerate}
	
then $M$ admits a pl isometric embedding into $\mathbb{E}^{3n-m}$.

\end{theorem}

\bibliographystyle{amsplain}

\begin{thebibliography}{10}

\bibitem {Akopyan} A. V. Akopyan, \textit{PL-analogue of Nash-Keiper theorem}, 
preliminary version (in Russian):  http://www.moebiuscontest.ru/files/2007/akopyan.pdf  \; \; www.moebiuscontest.ru

\bibitem {AT} A. V. Akopyan and A. S. Tarasov, \textit{A constructive proof of Kirszbraun's theorem},
Math. Notes, Vol. 84 (2008), no. 5-6, pp. 725-728

\bibitem {Brehm} U. Brehm, \textit{Extensions of distance reducing mappings to piecewise congruent mappings on $\R^m$},
J. Geom., Vol. 16 (1981), no. 2, pp. 187-193

\bibitem {Bridson} Martin Bridson, \textit{Geodesics and Curvature in Metric Simplicial Complexes}, 
Thesis, Cornell University (1991)

\bibitem {BH} Martin R. Bridson and Andre Haefliger, \textit{Metric Spaces of Non-Positive Curvature},
Springer (1991)

\bibitem {Burago} Dmitri Burago, Yuri Burago, and Sergei Ivanov, \textit{A Course in Metric Geometry},
American Mathematical Society (2001)

\bibitem {Donne} Enrico Le Donne, \textit{Lipschitz and path isometric embeddings of metric spaces},
Geom Dedicata, Springer Netherlands (2012)

\bibitem {Gromov2} Mikhael Gromov, \textit{Partial Differential Relations},
Springer-Verlag (1980), pp. 213

\bibitem {HY} J. Hocking and G. Young, \textit{Topology}, Dover
(1988), 214--215.

\bibitem {Krat} S. Krat, \textit{Approximation problems in length geometry},
Ph. D. Thesis, The Pennsylvania State University (2004)

\bibitem {Kuiper} Nicolaas H. Kuiper, \textit{On $C^1$-Isometric Imbeddings I, II},
Indag. Math., Vol. 17 (1955), pp. 545-556

\bibitem {Minemyer1} B. Minemyer, \textit{Simplicial Isometric Embeddings of Indefinite-Metric Polyhedra},
in preparation

\bibitem {Minemyer2} B. Minemyer, \textit{Isometric Embeddings of Polyhedra},
Ph. D. Thesis, The State University of New York at Binghamton (2013)

\bibitem {Nash1} John Nash, \textit{$C^1$ Isometric Imbeddings},
The Annals of Mathematics, Second Series, Vol. 60 (1954), pp. 383-396

\bibitem {Nash2} John Nash, \textit{The Imbedding Problem for Riemannian Manifolds},
The Annals of Mathematics, Second Series, Vol. 63, No. 1 (1956), pp. 20-63

\bibitem {PWZ} R. Penrose, J.H.C. Whitehead, and E.C. Zeeman, \textit{Imbedding of Manifolds in Euclidean Space},
The Annals of Mathematics, Vol. 73, No. 3 (1961), pp. 613-623

\bibitem {Petrunin1} A. Petrunin, \textit{On Intrinsic Isometries to Euclidean Space},
St. Petersburg Math. J., Vol. 22 (2011), No. 5, pp. 803-812

\bibitem {Petrunin2} A. Petrunin, Unpublished Lecture Notes, pp. 49-50 \\ http://www.math.psu.edu/petrunin/teach-old/m497C-2011-F-MASS/euclid2alexandrov.pdf

\bibitem {Zalgaller} V. A. Zalgaller, \textit{Isometric imbedding of polyhedra},
Dokl. Akad. Nauk SSSR (in Russian), Vol. 123 (1958), No. 4, pp. 599-601

\end{thebibliography}

\end{document}